\documentclass{article}

\usepackage{amsmath, amssymb, amsthm, mathtools}
\usepackage{mathabx,esint}
\usepackage{hyperref}
\usepackage{url}

\usepackage{geometry} 
\usepackage{setspace}

\usepackage[english]{babel}
\usepackage{csquotes} 

\usepackage{enumitem}

\usepackage{threeparttable}

\usepackage{xcolor}

\usepackage{version}
\usepackage[textsize=footnotesize]{todonotes}

\allowdisplaybreaks

\usepackage[style=alphabetic,maxnames=99,minnames=99]{biblatex}
\DeclareRedundantLanguages{English}{english}
\setlist[enumerate,1]{label=(\roman*)}

\newtheorem{theorem}{Theorem}[section]

\newtheorem{lemma}[theorem]{Lemma}
\newtheorem{proposition}[theorem]{Proposition}
\newtheorem{corollary}[theorem]{Corollary}

\theoremstyle{definition}
\newtheorem{definition}[theorem]{Definition}

\theoremstyle{remark}
\newtheorem{remark}[theorem]{Remark}

\newtheorem*{claim*}{Claim}

\numberwithin{equation}{section}

\usepackage{zref-clever}

\AddToHook{env/definition/begin}{%
\zcsetup{countertype={theorem=definition}}}
\AddToHook{env/example/begin}{%
\zcsetup{countertype={theorem=example}}}
\AddToHook{env/lemma/begin}{%
\zcsetup{countertype={theorem=lemma}}}
\AddToHook{env/corollary/begin}{%
\zcsetup{countertype={theorem=corollary}}}
\AddToHook{env/proposition/begin}{%
\zcsetup{countertype={theorem=proposition}}}
\AddToHook{env/remark/begin}{%
\zcsetup{countertype={theorem=remark}}}
\AddToHook{env/observation/begin}{\zcsetup{countertype={theorem=observation}}}

\zcRefTypeSetup{item}{name-sg= ,name-pl= }
\zcRefTypeSetup{equation}{name-sg= ,name-pl= }

\zcRefTypeSetup{observation}{name-sg=observation,name-pl=observations,Name-sg=Observation,Name-pl=Observations}

\zcRefTypeSetup{section}{cap}
\zcRefTypeSetup{figure}{cap}
\zcRefTypeSetup{table}{cap}
\zcRefTypeSetup{theorem}{cap}
\zcRefTypeSetup{definition}{cap}
\zcRefTypeSetup{example}{cap}
\zcRefTypeSetup{lemma}{cap}
\zcRefTypeSetup{corollary}{cap}
\zcRefTypeSetup{proposition}{cap}
\zcRefTypeSetup{remark}{cap}
\zcRefTypeSetup{observation}{cap}

\DeclareMathOperator{\supp}{\mathrm{supp}}

\newcommand{\ft}{\widehat}
\newcommand{\mc}{\mathcal}
\newcommand{\mr}{\mathrm}

\newcommand{\M}{\mc{M}}

\newcommand{\Mu}[1]{\mc{M}_{#1}^{\mr u}}
\newcommand{\Mc}[1]{\mc{M}_{#1}^{\mr c}}
\newcommand{\MZu}{\Mu\Z}
\newcommand{\MZc}{\Mc\Z}
\newcommand{\MRu}{\Mu\R}
\newcommand{\MRc}{\Mc\R}
\newcommand{\Cu}[1]{C^{#1,\mr u}}

\newcommand{\CZu}{\Cu\Z}

\newcommand{\CRu}{\Cu\R}

\DeclareMathOperator{\var}{Var}
\newcommand{\intav}[1]{\mathchoice {\mathop{\vrule width 6pt height 3 pt depth  -2.5pt
\kern -8pt \intop}\nolimits_{\kern -6pt#1}} {\mathop{\vrule width
5pt height 3  pt depth -2.6pt \kern -6pt \intop}\nolimits_{#1}}
{\mathop{\vrule width 5pt height 3 pt depth -2.6pt \kern -6pt
\intop}\nolimits_{#1}} {\mathop{\vrule width 5pt height 3 pt depth
-2.6pt \kern -6pt \intop}\nolimits_{#1}}}

\renewcommand{\leq}{\leqslant}
\renewcommand{\geq}{\geqslant}

\newcommand\R{\mathbb{R}}

\newcommand\Z{\mathbb{Z}}

\newcommand\N{\mathbb{N}}
\newcommand\BV{\mathrm{BV}}

\newcommand\loc{\mathrm{loc}}
\newcommand\cpt{\mathrm{c}}
\newcommand\cmod[2]{#1\%#2}
\newcommand\lm[1]{\left|#1\right|}
\newcommand\cm[1]{\left|#1\right|}

\newcommand{\intd}{\mathop{}\!\mathrm{d}}
\newcommand\av[3]{A^{#1}_{#2}#3}

\newcommand\sungyi[1]{\todo[color=cyan!30]{\textsc{Sung-Yi:} #1}}
\newcommand\julian[1]{\todo[color=blue!30]{\textsc{Julian:} #1}}

\newcommand\fulltitle{Sharp higher order regularity of discrete maximal functions}
\title{\fulltitle}

\newcommand\addressvt{Department of  Mathematics, Virginia Tech,  225 Stanger Street, Blacksburg, VA 24061-1026, USA}
\newcommand\addressvtshort{Virginia Tech}

\author{
Sung-Yi Liao\footnote{
\addressvt,
\texttt{sungyi@vt.edu}
},\ \ 
Jos\'e Madrid\footnote{
\addressvtshort,
\texttt{josemadrid@vt.edu}
},\ \ 
Eyvindur Palsson\footnote{
\addressvtshort,
\texttt{palsson@vt.edu}
},\ \ 
Julian Weigt\footnote{
Room 121, Leonardo Building, Abdus Salam International Centre for Theoretical Physics, Strada Costiera 11, 34151 Trieste, Italy,
\texttt{jweigt@ictp.it}
}\ \ 
}

\begin{document}

\maketitle

\begin{abstract}
We derive sharp $\ell^p(\Z)$ bounds for the $k$th derivative of the discrete uncentered maximal operator applied to characteristic functions $f:\Z\rightarrow\{0,1\}$ in the cases $k=0,1,2$. When $k=1,2$ these are the first sharp bounds for derivatives of a Hardy-Littlewood maximal function in continuous or discrete settings when $1<p<\infty$.
We also establish several lower bounds for $k\geq 3$ and for general functions $f:\Z\rightarrow\R$.
\end{abstract}

\begingroup
\begin{NoHyper}%
\renewcommand\thefootnote{}\footnotetext{%
}%
\addtocounter{footnote}{-1}%
\end{NoHyper}%
\endgroup

\section{Introduction}
\subsection{Operator norms for maximal operators}
For a locally integrable function $f:\mathbb{R}\to\mathbb{R}$, its centered Hardy-Littlewood maximal function $\MRc f$ is defined by
\begin{align}\label{def:centered}
\av\R I f
&:=
\frac{1}{\lm I}\int_I|f(y)|\intd y
,&
\MRc f(x)
&:=
\sup_{\emptyset\neq I=(x-r,x+r)}
\av\R I f
,
\end{align}
where $\lm I$ denotes the Lebesgue measure of the interval $I$.
The centered maximal operator was introduced by Hardy and Littlewood in 1930 \cite{hardy1930maximal} and has since become a fundamental tool in harmonic analysis.

By the Lebesgue differentiation theorem, \(f \leq \MRc f\) almost everywhere. Thus, the maximal operator \(\MRc\) provides a natural way to control the size of a function. A natural question is: how effective is this estimate? Hardy and Littlewood showed that \(\MRc\) is a bounded operator from \(L^p(\mathbb{R}^d)\) to \(L^p(\mathbb{R}^d)\) if and only if \(1<p\leq\infty\). For \(p=1\), the maximal operator maps \(L^1(\mathbb{R})\) to \(L^{1,\infty}(\mathbb{R})\).

Several results are known concerning the operator norm of the maximal operator. It is immediate that the operator norm of the centered maximal operator on \(L^\infty(\mathbb{R})\) is \(1\). Melas obtained the operator norm when \(\MRc:L^1(\mathbb{R})\rightarrow L^{1,\infty}(\mathbb{R})\) and we recall the statement next.
\begin{theorem}[\cite{melas2003best}]
For every $f\in L^1(\mathbb{R})$ and for every $\lambda>0$, we have
\[\lm{\{\MRc f>\lambda\}}\leq\frac{11+\sqrt{61}}{12\lambda}{\|f\|}_{L^1(\mathbb{R})},\]
and this is sharp.
\end{theorem}
The exact value of the operator norm of the centered Hardy-Littlewood maximal operator on \(L^p(\mathbb{R})\) remains unknown.

The uncentered Hardy-Littlewood maximal operator $\MRu$ considers averages over all intervals containing $x$ and not only those centered in $x$ like its centered variant $\MRc$, i.e.\ 
\begin{equation}\label{def:uncentered}
\MRu f(x)
:=
\sup_{I\ni x}
\av\R I f
,
\end{equation}
where the supremum is taken over all open intervals containing $x$.
Similar to the centered maximal operator, it maps $L^1(\R)$ to $L^{1,\infty}(\mathbb{R})$ and $L^p(\R)$ to $L^p(\R)$ for $1<p\leq\infty$ and its operator norm on $L^\infty(\mathbb{R})$ is 1.
For $1<p<\infty$ the exact value of its operator norm on $L^p(\R)$ was found by Grafakos and Montgomery.
\begin{theorem}[\cite{GM3}]\label{operator_norm_of_uncentered_maximal_operator_in_continuous_setting} 
\label{thm: GM}
For $1<p<\infty$,
the operator norm of the uncentered maximal operator $\MRu:L^p(\mathbb{R})\to L^p(\mathbb{R})$, is the unique positive solution of the equation
\begin{equation}\label{eq: GM}
(p-1)x^p-px^{p-1}-1=0.
\end{equation}
\end{theorem}

In this manuscript we study the discrete analog of the uncentered maximal operator. For $f:\mathbb{Z}\to\mathbb{R}$, and any nonempty interval $I\subset\Z$ we
define
\begin{align*}
\av\Z If
&:=
\frac{1}{\cm I}\sum_{i\in I}|f(i)|
,&
\MZu f(n)
&:=
\sup_{I\ni n}
\av\Z If
,
\end{align*}
where the supremum is taken over all discrete intervals $I$ containing $n$ and $\cm I$ denotes the counting measure of $I$.
The boundedness of \( \MZu \) on \( \ell^p(\mathbb{Z}) \) for \( 1 < p \leq \infty \) follows by an argument analogous to the continuous setting.
This was observed, for instance, in the book by Stein and Murphy \cite{stein1993harmonic}, though the theme dates back at least to the classic paper of Riesz \cite{Riesz28}, as discussed in \cite{SW99}.

\subsection{First derivative bounds for maximal operators}
Beyond size estimates, control of oscillation is another aspect that has lately attracted considerable attention.
In 1997, Kinnunen established the following oscillation estimate.
\begin{theorem}[\cite{kinnunen1997hardy}]
For every $1<p\leq\infty$, there exists $C_p<\infty$ such that 
\begin{equation}\label{eq:Kinnunen}
{\|(\MRc f)'\|}_{L^p(\mathbb{R})}\leq C_p{\|f'\|}_{L^p(\mathbb{R})}.
\end{equation}
for all functions $f\in W^{1,p}(\mathbb{R})$.
\end{theorem}

Regularity bounds of this type have been studied extensively in a variety of settings studied during the decades since.
Most attention has focused on the endpoint case $p=1$.
Surprisingly, although the centered maximal operator $\MRc$ does not map $L^1(\mathbb{R})$ into itself, the corresponding gradient bound \zcref{eq:Kinnunen} has nevertheless been proved for $p=1$ in several special cases.

At this endpoint, one can consider each of the two following quantities: the $L^1(\mathbb{R})$-norm $\|f'\|_{L^1(\mathbb{R})}$ of the weak derivative of an absolutely continuous function, and the variation $\var(f)$ of a function $f$ in the larger space $\BV(\mathbb{R})$ of functions of bounded variation.
Most of the results below are known in both formulations, although we state each result in the setting in which it first appeared.

Perhaps the most notable result is Kurka's proof of the endpoint gradient bound for the centered Hardy-Littlewood maximal operator in one dimension.

\begin{theorem}[\cite{Ku}]
There exists a constant $C<\infty$ such that for every function $f:\mathbb{R}\to\mathbb{R}$ of bounded variation
\begin{equation}\label{eq:Kurka}
\var(\MRc f)\leq C\var(f).
\end{equation}
\end{theorem}

For radially decreasing functions, equality in \zcref{eq:Kurka} is attained with $C=1$.
It is widely conjectured that \zcref{eq:Kurka} holds with $C=1$ in full generality, but this remains open.
Bilz and the fourth author proved this in the case of characteristic functions, where we have to consider the variation instead of the $L^1(\mathbb{R})$ norm of the weak derivative.

\begin{theorem}[\cite{BW}]
Let $f:\mathbb{R}\to\{0,1\}$ be a function of bounded variation.
Then
\[
\var(\MRc f)
\leq
\var(f)
.
\]
Equality holds if and only if $f$ is constant or the set $\{x\in\mathbb{R}\mid f(x)=1\}$ is a bounded interval of positive length.
\end{theorem}

For the uncentered maximal operator Aldaz and Pérez Lázaro proved the variation bound with constant 1, which is the smallest possible constant as can again be observed by considering radially decreasing functions.

\begin{theorem}[\cite{AP}]
\label{thm:bestc1p}
For any function $f:\mathbb{R}\to\mathbb{R}$ with bounded variation we have 
\[
\var(\MRu f)
\leq
\var(f)
.
\]
\end{theorem}

The best constant is also known in the endpoint $p=\infty$:

\begin{theorem}[\cite{aldaz2012optimal}]
For any function $f\in W^{1,\infty}(\mathbb{R})$ we have
\[{\|{(\MRu f)}'\|}_{L^\infty(\R)}\leq(\sqrt 2-1){\|f'\|}_{L^\infty(\R)}\]
and this bound is sharp.
\end{theorem}

For $1<p<\infty$ the optimal constant still remains open.
We may ask the same oscillation control question for the uncentered maximal operator in discrete setting.
For $k\in\N$ and a discrete function $f:\Z\rightarrow\R$ we define its $k$th derivative $f^{(k)}$ inductively by
\begin{align*}
f^{(0)}(n)&=f(n)
,&
f^{(k+1)}(n)&=f^{(k)}(n+1)-f^{(k)}(n)
.
\end{align*}
We also write $f'=f^{(1)}$, $f''=f^{(2)}$ etc.
Similar questions as above have been considered in this discrete setting. For $f:\mathbb{Z}\to\mathbb{R}$, sublinearity of $\MZu$ gives $(\MZu f)'\leq \MZu(f')$, and hence the discrete Hardy-Littlewood maximal theorem implies that for every $1<p\leq\infty$ there exists $C<\infty$ such that for all $f:\Z\rightarrow\R$ with $\|f'\|_{\ell^p(\Z)}<\infty$ we have
\[
{\|{(\MZu f)}'\|}_{\ell^p(\mathbb{Z})}\leq C{\|f'\|}_{\ell^p(\mathbb{Z})}.
\]
Bober, Carneiro, Hughes, and Pierce determined the smallest possible constant $C$ in the case $p=1$.
\begin{theorem}[\cite{BCHP}]
\label{thm: BCHP}
For all $f:\mathbb{Z}\to\mathbb{R}$ with $f'\in\ell^1(\mathbb{Z})$,
\[
\left\|{(\MZu f)}'\right\|_{\ell^1(\mathbb{Z})}\leq\|f'\|_{\ell^1(\mathbb{Z})}.
\]
Equality holds for $f=\chi_{\{0\}}$.
\end{theorem}
Gonzalez-Riquelme and the second author determined the smallest possible constant $C$ in the case $p=\infty$.
\begin{theorem}[\cite{GM2}]
\label{thm: GM2}
For all $f:\mathbb{Z}\to\mathbb{R}$ with $f'\in\ell^\infty(\mathbb{Z})$,
\[
\left\|{(\MZu f)}'\right\|_{\ell^\infty(\mathbb{Z})}\leq\frac{1}{2}\|f'\|_{\ell^\infty(\mathbb{Z})}.
\]
Equality holds for $f=\chi_{\{0\}}$.
\end{theorem}

For characteristic functions $f:\Z\rightarrow\{0,1\}$ we have $\|f'\|_{\ell^\infty(\mathbb{Z})}\leq1$ and thus by interpolating between \zcref{thm: BCHP,thm: GM2} for any $1\leq p\leq\infty$ we obtain
\begin{equation}
\label{eq:charfpinterpolation}
{\|(\MZu f)'\|}_{\ell^p(\mathbb{Z})}\leq {\|(\MZu f)'\|}_{\ell^1(\mathbb{Z})}^\frac{1}{p}{\|(\MZu f)'\|}_{\ell^\infty(\mathbb{Z})}^{1-\frac{1}{p}}\leq {\|f'\|}_{\ell^1(\mathbb{Z})}^\frac{1}{p}2^{-\left(1-\frac{1}{p}\right)}{\|f'\|}_{\ell^\infty(\mathbb{Z})}^{1-\frac{1}{p}}\leq2^{\frac{1}{p}-1}\|f'\|_{\ell^p(\mathbb{Z})}
.
\end{equation}

\subsection{Second and higher derivative bounds for maximal operators}
For higher derivatives it has been folklore that for $1\leq p\leq\infty$ and $k\geq3$, as well as $1<p\leq\infty$ and $k=2$, the corresponding bound fails, i.e.\ there exist $f_1,f_2,\ldots:\mathbb{R}\rightarrow\mathbb{R}$ with
\begin{align*}
{\|f_n^{(k)}\|}_{L^p(\mathbb{R})}&\leq1,&
{\|{(\MRc f_n)}^{(k)}\|}_{L^p(\mathbb{R})}&\rightarrow\infty,&
{\|{(\MRu f_n)}^{(k)}\|}_{L^p(\mathbb{R})}&\rightarrow\infty
\end{align*}
because the maximal function of a smooth function may have corners. For the precise argument for the uncentered maximal operator $\MRu$, see \cite{paper2}. As for the endpoint case $(k,p)=(2,1)$, the fourth author found examples with
\begin{align}
\label{counterexamplec21}
{\|f_n''\|}_{L^1(\mathbb{R})}&\leq1
,&
{\|{(\MRu f_n)}''\|}_{L^1(\mathbb{R})}&\rightarrow\infty
\end{align}
in \cite{W}.
To be precise, they are proven for the uncentered maximal operator that does not take the absolute value in the integral.
In our forthcoming paper \cite{paper2}, we will show how they can be adapted to also apply to the classical maximal function defined in \zcref{def:uncentered} with the absolute value, and moreover extend the result to the discrete setting of functions on $\mathbb{Z}$.
Curiously, when considering only characteristic functions on $\mathbb{Z}$, the picture changes significantly. For the second derivative, Temur proved the following.
\begin{theorem}[\cite{Te2}]
\label{thm: Temur22}
Let $1\leq p\leq\infty$.
Then for all $S\subset\Z$, we have
\begin{equation}\label{eq:Temur22}
\|(\MZu\chi_S)''\|_{\ell^p(\mathbb{Z})}\leq 2^{1-\frac{1}{p}}3^\frac{1}{p}\|\chi_S''\|_{\ell^p(\mathbb{Z})}.
\end{equation}
\end{theorem}
Temur and \"{O}zcan proved the following bounds for higher derivatives:
\begin{theorem}[\cite{TO}]
\label{thm: Temur25}
Let $1\leq p\leq\infty$ and $k\geq 1$.
Then there exists $C\geq0$ such that for all $S\subset\Z$, we have
\[
\left\|{(\MZu\chi_S)}^{(k)}\right\|_{\ell^p(\mathbb{Z})}\leq C\|\chi_S^{(k)}\|_{\ell^p(\mathbb{Z})}.
\]
\end{theorem}

We now retrace the proof of \zcref{thm: Temur25} in \cite{TO} in order to recover an explicit value for $C$.
They show that for all $S\subseteq\mathbb{Z}$, 
\[
{\|\chi_S'\|}_{\ell^p(\mathbb{Z})}\leq\min\left\{{(2k+1)}^{\frac{1}{p}},2^{\frac{k}{p}}\right\}{\left\|\chi_S^{(k)}\right\|}_{\ell^p(\mathbb{Z})}.
\]
Thus, by \zcref{eq:charfpinterpolation} and the triangle inequality,
\begin{equation}\label{eq:explicit_bound_from_TO25}
\begin{split}
\left\|{(\MZu\chi_S)}^{(k)}\right\|_{\ell^p(\mathbb{Z})}&\leq 2^{k-1}{\|(\MZu \chi_S)'\|}_{\ell^p(\mathbb{Z})}\leq2^{k-1+(\frac{1}{p}-1)}{\|\chi_S'\|}_{\ell^p(\mathbb{Z})}\\
&\leq\min\left\{2^{k+\frac{1}{p}-2}{(2k+1)}^{\frac{1}{p}},2^{k+\frac{k+1}{p}-2}\right\}{\left\|\chi_S^{(k)}\right\|}_{\ell^p(\mathbb{Z})}.
\end{split}
\end{equation}

\subsection{Table of known results}\label{sec:tables}

\begin{definition}
For $k\in\N_0$ and $1\leq p\leq\infty$ denote by $C_{k,p}\geq0$ the smallest number such that for all subset $S\subseteq\mathbb{Z}$ with \({\left\|\chi_S^{(k)}\right\|}_{\ell^p(\Z)}<\infty\) we have
\begin{equation}\label{eq: goal}
{\left\|{(\MZu\chi_S)}^{(k)}\right\|}_{\ell^p(\mathbb{Z})}\leq C_{k,p}{\left\|\chi_S^{(k)}\right\|}_{\ell^p(\mathbb{\mathbb{Z}})}
,
\end{equation}
and similarly define $C_{0,p,\infty}$ for the weak bound
\[
\sup_{\lambda>0}
\lambda^{\frac1p}
\cm{\{n\in\Z:\MZu\chi_S(n)>\lambda\}}
\leq
C_{0,p,\infty}
{\left\|\chi_S\right\|}_{\ell^p(\mathbb{\mathbb{Z}})}
.
\]
\end{definition}

\begin{table}[h]
\centering
\begin{threeparttable}
\begin{tabular}{|c|cccc|}
\hline
$p$\textbackslash$k$ & 0 & 1 & 2 & $\geq3$
\\
\hline
1 & $\infty$ & 1  & $\leq3$  & $\leq2^{k-1}(2k+1)$
\\
$1<p<\infty$ & $<\infty$ & $\leq2^{\frac1{p}-1}$  & $\leq\min\left\{2^{1-\frac1p}3^{\frac1p},2^\frac{3}{p}\right\}$  & $\leq2^{k+\frac{1}{p}-2}{(2k+1)}^{\frac{1}{p}}$ 
\\
$\infty$ & 1 & $\frac12$  & $\leq1$  & $\leq2^{k-2}$  \\\hline
\end{tabular}
\end{threeparttable}
\caption{Known values $C_{k,p}$ for the uncentered maximal operator in the discrete characteristic function setting.}
\label{table:ckp}
\end{table}

In \zcref{table:ckp} the entries for $k=1$ equal \zcref{eq:charfpinterpolation}, the entries for $k=2$ equal the minimum of \zcref{eq:Temur22,eq:explicit_bound_from_TO25} and the entries for $k\geq3$ equal \zcref{eq:explicit_bound_from_TO25}.



\section{Main Results}
As observed in \zcref{table:ckp}, sharp derivative bounds, particularly for discrete maximal operators, remain largely unknown in the existing literature.
In this manuscript, we improve the bounds for $C_{k,p}$, i.e.\ for characteristic functions, and update \zcref{table:ckp} to \zcref{table:ckpnew}.


We call $(\chi_{S_m})_{m\in\N}$ an extremizing sequence for $C_{k,p}$ if
\[C_{k,p}=\lim_{m\to\infty}\frac{{\|(\MZu\chi_{S_m})^{(k)}\|}_{\ell^p(\Z)}}{{\|\chi_{S_m}^{(k)}\|}_{\ell^p(\Z)}}.\]
\begin{theorem}\label{thm:c0p}
For all $1\leq p\leq\infty$ we have
\begin{align*}
C_{0,p}
&=
\left(\frac{p+1}{p-1}\right)^{\frac1p}
,&
C_{0,p,\infty}
&=
2^{1/p},
\end{align*}
with the usual limiting interpretations at $p=1,\infty$. The function $\chi_{\{0\}}$ is an extremizer for $C_{0,p,\infty}$, and with $S_{m}:=\{1,2,\dots,m\}$, the sequence $(\chi_{S_m})_{m\in\N}$ is extremizing for $C_{0,p}$.
\end{theorem}

\begin{remark}
For every function $f:\mathbb{Z}\to\mathbb{R}$, we have
\(
\MZc f(n)\leq \MZu f(n)
\)
for all $n\in\mathbb{Z}$. Thus, for every $S\subseteq\mathbb{Z}$,
\[
{\|\MZc\chi_S\|}_{\ell^p(\mathbb{Z})}
\leq {\|\MZu\chi_S\|}_{\ell^p(\mathbb{Z})}
\leq \left(\frac{p+1}{p-1}\right)^{\frac{1}{p}}
{\|\chi_S\|}_{\ell^p(\mathbb{Z})}.
\]
It remains unclear whether this bound for the centered maximal operator is sharp.
\end{remark}

\begin{theorem}\label{thm:c1p}
For all $1\leq p\leq \infty$ we have
\[C_{1,p}={\left(\sum_{n=1}^\infty {\left(\frac{1}{n}-\frac{1} {n+1}\right)}^p\right)}^{\frac{1}{p}}\]
with the usual limiting interpretations at $p=\infty$. The function $\chi_{\{0\}}$ is an extremizer.
\end{theorem}

This gives the sharp bound for \zcref{eq:charfpinterpolation}. To the best of our knowledge, this is the first sharp $\ell^p(\Z)$-bound for the gradient of a Hardy-Littlewood maximal function for any $1<p<\infty$.
The only previously known sharp bound for the gradient of maximal functions at an exponent $p\neq1,\infty$ is due to \cite{CS} who, among other results, find that $1$ is best constant at exponent $p=2$ for the convolution-type maximal operator associated to the Gauss kernel in dimensions $d\geq3$.

As a consequence of \zcref{thm:c1p} we can now improve \zcref{eq:explicit_bound_from_TO25} to
\begin{equation}\label{eq:updated_bound_from_TO25}
\left\|{(\MZu f)}^{(k)}\right\|_{\ell^p(\mathbb{Z})}\leq 2^{k-1}{\left(\sum_{n=1}^\infty {\left(\frac{1}{n}-\frac{1}{n+1}\right)}^p\right)}^\frac{1}{p}{(2k+1)}^{\frac{1}{p}}{\left\|f^{(k)}\right\|}_{\ell^p(\mathbb{Z})},
\end{equation}
thereby improving on \zcref{thm: Temur25}, the main result of \cite{TO}.

Let $\CZu_{1,p}$ denote the smallest constant such that 
for every $f:\mathbb{Z}\to\mathbb{R}$ with ${\left\|f'\right\|}_{\ell^p(\mathbb{Z})}<\infty$
one has
\begin{equation}\label{eq:def:czu}
{\left\|{(\MZu f)}'\right\|}_{\ell^p(\mathbb{Z})}\leq\CZu_{k,p}{\left\|f'\right\|}_{\ell^p(\mathbb{Z})}.
\end{equation}

For $p=1,\infty$, the results of \cite{BCHP,GM2}, stated in \zcref{thm: BCHP,thm: GM2}, imply $\CZu_{1,p}=C_{1,p}$. In both cases, $\chi_{\{0\}}$ gives equality in \zcref{eq:def:czu}. Surprisingly, in \zcref{thm:czu1p_c1p}, we show that this equality fails for every $1<p<\infty$.

\begin{theorem}\label{thm:czu1p_c1p}
Let $1<p<\infty$. Then, $\CZu_{1,p}>C_{1,p}$.
\end{theorem}

\begin{theorem}\label{thm:c2p}
For all $1\leq p\leq \infty$ we have
\[C_{2,p}=\frac{1}{2}.\] 
Defining $S_{m}:=\{0,3,6,\dotsc,3(m-1)\}$, the sequence $(\chi_{S_m})_{m\in\N}$ is extremizing.
\end{theorem}
This sharpens \zcref{thm: Temur22}, the main result from \cite{Te2}.

\begin{remark}
At $p=1,\infty$, the function $\chi_{\{0\}}$ is also an extremizer for $C_{2,p}$; however, it is not an extremizer at $1<p<\infty$.
\end{remark}

\begin{theorem}\label{thm:ckp}
For all $k\geq 0$, $1\leq p<\infty$ we have
\[C_{k,p}\geq\frac{1}{2}.\]
This lower bound is witnessed by the sequence $(\chi_{S_m})_{m\in\N}$ from \zcref{thm:c2p}.
\end{theorem}

As $k$ grows, we have the following improvement on the previous bound.

\begin{theorem}\label{thm:ckpgeq}
For all $k\geq0$ and $1\leq p\leq\infty$ we have
\[C_{k,p}\geq\frac{1}{12\cdot 8^{\left|\frac1p-\frac12\right|}}\left(\frac{2}{|1+e^{\frac{\pi i}4}|}\right)^k.\]
In particular, $C_{k,p}$ grows at least exponentially in $k$, uniformly in $p$.
This bound is witnessed by the sequence $(\chi_{B_m})_{m\in\mathbb{N}}$, where
\begin{align*}
B_1
&=
\{0,1,2,5\}
,&
B_{m+1}
&=
B_m\cup(B_1+8m)
.
\end{align*}
\end{theorem}

\begin{remark}
Somewhat surprisingly, \zcref{thm:ckpgeq} shows that the exponential growth in $k$ appearing in \zcref{thm: Temur25} is sharp, although only in the weak sense that both the lower and upper bounds are exponential in $k$, with different bases strictly larger than one,
\begin{equation}
\label{exponentialestimates}
1<1.0824\approx\frac{2}{|1+e^{\frac{\pi i}{4}}|}\leq\liminf_{k\to\infty} C_{k,p}^{\frac1k}\leq\limsup_{k\to\infty} C_{k,p}^{\frac1k}\leq2.
\end{equation}
\end{remark}

We can now update \zcref{table:ckp} to \zcref{table:ckpnew}.
\begin{table}[h]
\centering
\begin{threeparttable}
\begin{tabular}{|c|cccc|}
\hline
$p$\textbackslash$k$ & 0 & 1 & 2 & $\geq3$
\\
\hline
1 & $\infty$ & 1 & \fbox{$\frac12$} & \fbox{$\frac12\leq C_{k,1}\sim C^k$}
\\
$1<p<\infty$
&
\fbox{$\left(\frac{p+1}{p-1}\right)^{\frac1p}$}
&
\fbox{${\left(\sum_{n=1}^\infty {\left(\frac{1}{n}-\frac{1} {n+1}\right)}^p\right)}^{\frac{1}{p}}$}
&
\fbox{$\frac12$}
&
\fbox{$\frac12\leq C_{k,p}\sim C^k$}
\\
$\infty$ & 1 & $\frac12$ & \fbox{$\frac12$} & \fbox{$\frac12\leq C_{k,\infty}\sim C^k$}
\\
\hline
\end{tabular}
\begin{tablenotes}
\footnotesize
\item[] Here \(C_{k,p}\sim C^k\) is not meant as asymptotic equivalence up to multiplicative constants. It indicates exponential growth in \(k\), in the sense of \eqref{exponentialestimates}.
\end{tablenotes}
\caption{The updated values for \(C_{k,p}\) are boxed.}
\label{table:ckpnew}
\end{threeparttable}
\end{table}

\paragraph{Proof strategies}

\begin{enumerate}
\item
The proof of \zcref{thm:c0p} follows the same lines as \zcref{operator_norm_of_uncentered_maximal_operator_in_continuous_setting}, the corresponding result for $f\in L^p(\R)$ from \cite{GM3}.
\item
\zcref[S]{thm:czu1p_c1p} and the upper bounds in \zcref{thm:c1p,thm:c2p}
are obtained with the help of various convexity arguments such as Jensen's inequality or Karamata's inequality; see \zcref{Karamata}.
\item
The lower bounds in \zcref{thm:c2p,thm:ckp,thm:ckpgeq} are all witnessed by the functions $\chi_{\{0\}}$ on the circle $\Z_3=\{0,1,2\}$ and $\chi_{\{0,1,2,5\}}$ on the circle $\Z_8$ respectively. They extend from the periodic setting to $\mathbb{Z}$ by \zcref{lem:discretevsperiodic}.
\item
We initially discovered the example $\chi_{\{0,1,2,5\}}$ on the circle of length $8$ through numerical experiments; the source code is available at \cite{discretemf}.
More precisely, for a given circle, we enumerated all characteristic functions $\chi_S$, computed its maximal functions $\MZu\chi_S$ and evaluated $\|\MZu \chi_S^{(k)}\|_{\ell^p(\mathbb{Z})}/\|\chi_S^{(k)}\|_{\ell^p(\Z)}$ for selected values of $k$ and $p$.
Since the number of characteristic functions on a circle grows exponentially in its length, this approach becomes prohibitively expensive quickly, limiting us to circles of length at most about $30$.
Within these limitations, our numerical experiments suggest that for any $1\leq p\leq \infty$ the map $k\mapsto C_{k,p}$ is nondecreasing for $k\geq2$. 
Note that $k\mapsto C_{k,p}$ is decreasing for $k$ on $\{0,1,2\}$.
Surprisingly, we have not been able to experimentally improve on the lower bound $C_{k,p}\geq1/2$, witnessed by $\chi_{\{0\}}$ in \zcref{thm:ckp}, for any $k,p$ with $3\leq k\leq 6$ and $1\leq p\leq\infty$.
\item
The proof of the lower bound \zcref{thm:ckpgeq} relies on the discrete Fourier transform on the circle.
To the best of our knowledge this marks only the second instance of a successful application of Fourier methods to regularity of maximal functions, after \cite{BRS}.
In comparison to \cite{BRS}, our Fourier methods are significantly more elementary, however continue to apply at the endpoint $p=1$.
\end{enumerate}


\section{Preliminaries}\label{sec:preliminaries}
In this section, we will introduce some basic but useful observations related to the uncentered maximal operators in the discrete setting.

\subsection{Karamata's inequality}

\begin{lemma}[Karamata's inequality]\label{Karamata}
Let $g:\mathbb{R}\to\mathbb{R}$ be a convex function, $d\in\mathbb{N}$, and assume for $i=1,\ldots,d$ and $s_i,t_i\in\mathbb{R}$ the following hold:
\begin{itemize}
\item $s_1\geq s_2\geq\dotsb\geq s_d$ and $t_1\geq t_2\geq\dotsb\geq t_d$;
\item $s_1+\dotsb+s_i\geq t_1+\dotsb+t_i$ for all $i\in\{1,\dotsc,d-1\}$;
\item $s_1+\dotsb+s_d=t_1+\dotsb+t_d$.
\end{itemize}
Then, 
\begin{equation}\label{ineq:Karamata}
g(s_1)+\dotsb+g(s_d)\geq g(t_1)+\dotsb+g(t_d).
\end{equation}
Moreover, if $g$ is strictly convex, then equality in \zcref{ineq:Karamata} holds if and only if $s_i=t_i$ for all $i=1,\dotsc,d$.
\end{lemma}
\begin{remark}\label{infinite_Karamata}
\zcref[S]{Karamata} also holds for $d=\infty$, i.e.\ infinite sequences $(s_i)_{i=1}^\infty$ and $(t_i)_{i=1}^\infty$.
\end{remark}

\subsection{Basic properties of maximal functions}

\begin{lemma}\label{flat_bound}
For any function $f:\mathbb{Z}\to\mathbb{R}$ and $n\in\mathbb{Z}$, we have 
\(
\MZu f(n)
\geq
\frac12
\MZu f(n\pm1)
.
\)
In particular, if $S\subset\mathbb{Z}$ and $f=\chi_S$ then
\begin{align*}
\|
(\MZu f)'
\|_{\ell^\infty(\Z)}
&\leq
\frac12
,&
\|
(\MZu f)''
\|_{\ell^\infty(\Z)}
&\leq
1
.
\end{align*}
\end{lemma}

\begin{proof}
Let $n\in\mathbb{Z}$ and $I\ni n$ be an interval.
Then there exists an interval $J\supset I$ with $n+1\in J$ and \(\cm J\leq\cm I+1\leq 2\cm I\).
Thus \(\av \Z J f\geq\av \Z I f/2\) and
we can conclude $\MZu f(n)\geq\frac12\MZu f(n+1)$.
We can prove $\MZu f(n)\geq\frac12\MZu f(n-1)$ similarly.
\end{proof}

\begin{lemma}\label{lem: convexity}
Let $S\subset\Z$.
\begin{enumerate}
\item
\label{it:concavity}
If $n+1\in S$, then ${(\MZu\chi_S)}''(n)\leq0$.
\item
\label{it:convexity}
If $n+1\in\mathbb{Z}\setminus S$, then ${(\MZu\chi_S)}''(n)\geq0$.
\end{enumerate}
\end{lemma}

\zcref[S]{lem: convexity} is equivalent to the following \zcref[noref,nocap]{cor: convexity}:
\begin{corollary}\label{cor: convexity}
Let $S\subset\Z$.
\begin{enumerate}
\item
If $I$ is a connected component of $S$ then $\MZu\chi_S$ is concave on $(I-1)\cup I\cup(I+1)$.
\item
If $I$ is a connected component of $\Z\setminus S$ then $\MZu\chi_S$ is convex on $(I-1)\cup I\cup(I+1)$.
\end{enumerate}
\end{corollary}

\begin{proof}[Proof of \zcref{lem: convexity}]
Recall
\[(\MZu\chi_S)''(n) = \MZu\chi_S(n) + \MZu\chi_S(n+2) - 2\MZu\chi_S(n+1).\]
Let $n+1\in S$.
Then \zcref{it:concavity} follows from $\MZu\chi_S(n+1)=1$ and \(\max\{\MZu\chi_S(n),\MZu\chi_S(n+2)\}\leq1\).

It remains to consider $n+1\in\mathbb{Z}\setminus S$.
If $\MZu \chi_S(n+1)\leq \min\{\MZu \chi_S(n),\MZu \chi_S(n+2)\}$ then \zcref{it:convexity} follows similarly to the previous argument.
Thus, it remains to consider $\MZu \chi_S(n+1)>\min\{\MZu \chi_S(n),\MZu \chi_S(n+2)\}$, and by symmetry it suffices to consider $\MZu \chi_S(n)<\MZu \chi_S(n+1)$.
This means in the supremum defining $\MZu \chi_S(n+1)$ it suffices to consider intervals to the right:
\[\MZu \chi_S(n+1)=\sup_{r\geq 0}\av{\Z}{[n+1,n+1+r]} \chi_S.\]
For each $r\geq 0$, we have $\tfrac{1}{r+1}\le \tfrac12(\tfrac1r+\tfrac1{r+2})$, and thus using $\chi_S(n+1)=0$, we have
\begin{equation*}
\begin{split}
\av{\Z}{[n+1,n+1+r]}\chi_S&=\frac{r}{r+1}\sum_{k=1}^{r}\chi_S(n+1+k)\\
&\leq\frac{1}{2}\left(\frac{1}{r}\sum_{k=1}^{r}\chi_S(n+1+k)+\frac{1}{r+2}\sum_{k=1}^{r}\chi_S(n+1+k)\right)\\
&\leq\frac{1}{2}\left(\av{\Z}{[n+2,n+r+1]}\chi_S+\av{\Z}{[n,n+r+1]}\chi_S\right).
\end{split}
\end{equation*}
Taking the supremum over $r\geq 0$ gives 
\[\MZu\chi_S(n+1)\leq\sup_{r\geq0}\frac{1}{2}\left(\av{\Z}{[n+2,n+r+1]}\chi_S+\av{\Z}{[n,n+r+1]}\chi_S\right)\leq\frac{1}{2}\MZu\chi_S(n)+\frac{1}{2}\MZu\chi_S(n+2)\]
and thus $(\MZu \chi_S)''(n)\geq 0$.
\end{proof}

\begin{lemma}\label{lem: good balls}
Let $f:\Z \to[0,\infty)$ in $\ell^{\infty}(\Z)$.
If for all bounded intervals $I\ni n$ we have \(\MZu f(n)>\av {\mathbb{Z}}If\) then $n$ has to be a global minimum of $\MZu f$.

\end{lemma}

\begin{proof} 
By assumption for $i=1,2,\ldots$ there must exist $r_i,s_i \in \N_0$ with $(r_i+s_i) \to \infty$ and $\av{\mathbb{Z}}{[n-r_i,n+s_i]}f\to \MZu f(n)$ as $i \to \infty$.
That means for any $m \in \Z$ we have
\[
\av{\mathbb{Z}}{[m-r_j,m+s_j]}f\geq \av{\mathbb{Z}}{[n-r_j,n+s_j]}f- \frac{2 \|f\|_{\ell^\infty(\Z)} |n-m|}{(1+r_j + s_j)},
\]
and letting $i\rightarrow\infty$ we can conclude $\MZu f(m)\geq\MZu f(n)$.

\end{proof}
\begin{remark}\label{obs:about_polynomial}
Let $f:\mathbb{Z}\rightarrow\mathbb{R}$.
\begin{enumerate}
\item \label{obs: derivative is polynomial}
If $f'$ is a polynomial of order $k$ then $f$ is a polynomial of order $k+1$ \cite[Section 2.6]{GrahamKnuthPatashnik1994ConcreteMathematics2e}.
\item \label{obs: k-th derivative vanishes}
Let $N\in\Z$ and $k\in\N$.
Then $f^{(k)}$ vanishes on $\mathbb{Z}\cap [N,\infty)$ if and only if $f$ agrees with a polynomial of degree at most $k-1$ on $\mathbb{Z}\cap[N,\infty)$.

In particular, if $f=\chi_S$ for some $S\subset\Z$ then $f^{(k)}$ vanishes on $[N,\infty)$ if and only if one of $S$ and $\mathbb{Z}\setminus S$ contains $[N,\infty)$.
\end{enumerate}
\end{remark}

\begin{lemma}
\label{lem:c1pboundedsuffices}
Let $1\leq p<\infty$ and $k\geq0$.
Then
\[
C_{k,p}
=
\sup\left\{
\frac{{\|{(\MZu\chi_S)}^{(k)}\|}_{\ell^p(\Z)}}{{\|\chi_S^{(k)}\|}_{\ell^p(\Z)}}
:
0<\cm{S}<\infty\right\},
\]
i.e.\ it suffices to consider nonempty bounded subsets $S\subseteq\mathbb{Z}$ in \zcref{eq: goal}. 
\end{lemma}

\begin{proof}
First of all, if \(0<\cm{S}<\infty\) then by \zcref{obs:about_polynomial}~\zcref{obs: k-th derivative vanishes} we may divide \zcref{eq: goal} by \({\|\chi_S^{(k)}\|}_{\ell^p(\Z)}>0\) and the above supremum on the right hand side is well defined and bounded by $C_{k,p}$.

For the reverse inequality, if $\cm S=0$ or ${\|\chi_S^{(k)}\|}_{\ell^p(\Z)}=\infty$ then \zcref{eq: goal} holds for any value of $C_{k,p}>0$, which means we do not need to consider these cases.

If ${\|\chi_S^{(k)}\|}_{\ell^p(\Z)}<\infty$ then since \(\chi_S^{(k)}(n)\in\mathbb{Z}\) exists $N>0$ such that $\chi_S^{(k)}(n)=0$ for all $|n|>N$.
By \zcref{obs:about_polynomial}~\zcref{obs: k-th derivative vanishes} this means there exist $a,b\in\{0,1\}$ such that $\chi_S=a$ is constant on $[N,\infty)$ and similarly $\chi_S=b$ on $(-\infty,-N-k]$.
If $1\in\{a,b\}$ then $\MZu\chi_S=1$ everywhere, which means ${\|{(\MZu\chi_S)}^{(k)}\|}_{\ell^p(\Z)}=0$ so that \zcref{eq: goal} holds for any value of $C_{k,p}$.
We can conclude that it suffices to consider nonempty bounded $\cm S<\infty$ in the definition of $C_{k,p}$.
\end{proof}

\subsection{The discrete circle}
\label{subsec:basicnotions}
While the main setting we consider is the discrete line, $\Z$, we will use its connection to the discrete circle throughout the manuscript.
\begin{definition}
For a fixed positive integer $N$, define the discrete circle of order $N$ 
\[\mathbb{Z}_N:=\{0,1,\dotsc,N-1\}\]
equipped with the counting measure. For $n\in\Z$ we define
\[\cmod nN:=\min\{n+kN\geq0:k\in\mathbb{Z}\}.\]
Note that $\cmod nN\in\mathbb{Z}_N$ for all $n\in\mathbb{Z}$.
Our set of intervals $\mathcal I$ on $\Z_N$ is the set of images of intervals in $\Z$ under the map $n\mapsto\cmod nN$. Since $\mathbb{Z}_N$ has a cyclic structure, we call $\mathbb{Z}_N$ a discrete circle below. 
\end{definition}
Similar to $\MZu$, we can also define the uncentered maximal operator $\Mu{\mathbb{Z}_N}$ on $\mathbb{Z}_N$:
\begin{equation*}
\begin{split}
\Mu{\mathbb{Z}_N}f(n):=&\sup_{s,t\geq0}\frac{1}{s+t+1}\sum_{i=-s}^tf(\cmod{(n+i)}{N})\\
\end{split}
\end{equation*}
One may verify that it in fact suffices to consider $t+s+1\leq 2N$ in the above supremum.

For $k=0,1,2,\ldots$ we inductively define the $k$th derivative $f^{(k)}$ of $f:\Z_N\rightarrow\R$ by
\begin{align*}
f^{(0)}(n)
&:=
f(n)
,&
f^{(k+1)}(n)
&:=
f^{(k)}(\cmod{(n+1)}N)-f^{(k)}(n)
.
\end{align*}

Let $k\geq1$ and $N\geq1$.
Then, if $f:\Z_N\rightarrow\R$ is constant we have ${\|f^{(k)}\|}_{\ell^p(\mathbb{Z}_N)}=0$.
Next, we are going to show the reverse implication, cf.\ \zcref{obs:about_polynomial}~\zcref{obs: k-th derivative vanishes} and \zcref{lem:c1pboundedsuffices}.

\begin{lemma}\label{lem:modN_kth_derivative_non_vanishing}
Let $N\in\N$, $k\in\mathbb{N}$, $1\leq p\leq\infty$ and $f:\mathbb{Z}_N\to\mathbb{R}$ with ${\|f^{(k)}\|}_{\ell^p(\mathbb{Z}_N)}=0$.
Then $f$ is constant.
\end{lemma}

\begin{proof}
We will prove by induction on $l=0,1,\ldots,k$ that $f^{(k-l)}$ is constant.
In particular, this then holds for $f^{(0)}=f$. 

By assumption the induction conclusion holds for $l=0$.
So, let $0<l\leq k$ and assume there exists $a\in\mathbb{R}$ with $f^{k-(l-1)}=a$.
Then there exists $b\in\mathbb{R}$ such that for every $n\in\mathbb{N}$ we have $f^{(k-l)}(n)=an+b$.
Since $f^{(k-l)}$ is a periodic function this requires $a=0$ which means $f^{(k-l)}$ is constant.
This concludes the induction.
\end{proof}

The following \zcref[nocap,noref]{lem:discretevsperiodic} connects the setting of the discrete circle back to $\mathbb{Z}$. For $N\in\mathbb{N}$, $a\in\mathbb{N}$ and $f:\Z_N\rightarrow\mathbb{R}$ define $f_a:\mathbb{Z}\rightarrow\mathbb{R}$ by
\[
f_a(n)
=
f(\cmod nN)
\chi_{[0,aN)}(n)
.
\]

\begin{lemma}
\label{lem:discretevsperiodic}
Let $k\in\mathbb{N}$, $1\leq p\leq\infty$, $N\in\mathbb{N}$ and $f:\Z_N\rightarrow\mathbb{R}$ be nonconstant.
Then
\[
\liminf_{a\rightarrow\infty}
\frac{
\|
(\MZu f_a)^{(k)}
\|_{\ell^p(\mathbb{Z})}
}{
\|
f_a^{(k)}
\|_{\ell^p(\mathbb{Z})}
}
\geq
\frac{
\|
(\Mu{\mathbb{Z}_N}f)^{(k)}
\|_{\ell^p(\mathbb{Z}_N)}
}{
\|
f^{(k)}
\|_{\ell^p(\mathbb{Z}_N)}
}
.
\] 

\end{lemma}

\begin{proof}
If $p=\infty$, everything reduces to the local behavior. Thus, the statement holds clearly.

Assume that $p<\infty$. By \zcref{lem:modN_kth_derivative_non_vanishing} we have $\|f^{(k)}\|_{\ell^p(\mathbb{Z}_N)}>0$.

In the following ranges the following equalities hold:
\begin{align*}
2N\leq n&<(a-2)N:
&
\MZu f_a(n)
&=
\Mu{\mathbb{Z}_N}f(\cmod nN)
,\\
0\leq n&<aN-k:
&
f_a^{(k)}(n)
&=
f^{(k)}(\cmod nN)
,\\
2N\leq n&<(a-2)N-k:
&
(\MZu f_a)^{(k)}(n)
&=
(\Mu{\mathbb{Z}_N}f)^{(k)}(\cmod nN)
.
\end{align*}
For $a\geq k/N$ we have
\begin{align*}
\lfloor a-k/N\rfloor{\|f^{(k)}\|}_{\ell^p(\mathbb{Z}_N)}^p
&=
\|f_a^{(k)}\|_{\ell^p([0,N\lfloor a-k/N\rfloor-1])}^p
\leq
\|f_a^{(k)}\|_{\ell^p([0,aN-k-1])}^p
\leq
\|f_a^{(k)}\|_{\ell^p([0,aN-1])}^p
\\
&=
a{\|f^{(k)}\|}_{\ell^p(\mathbb{Z}_N)}^p.
\end{align*}
and thus,
\[
\lim_{a\to\infty}\frac{\|f_a^{(k)}\|_{\ell^p([0,aN-k-1])}^p}{a\|f^{(k)}\|_{\ell^p(\mathbb{Z}_N)}^p}
=1.
\]
Since $f$ is bounded and $f_a^{(k)}=0$ on $\Z\setminus(-k,aN)$, we have
\[
\limsup_{a\rightarrow\infty}
\|f_a^{(k)}\|_{\ell^p(\mathbb{Z}\setminus[0,aN-k-1])}^p
<\infty.
\]
Therefore,
\[
\lim_{a\to\infty}
\frac{
\|f_a^{(k)}\|_{\ell^p(\mathbb{Z})}^p
}{
a\|f^{(k)}\|_{\ell^p(\mathbb{Z}_N)}^p
}
=1
.
\]
For the maximal function one can similarly show
\begin{align*}
\liminf_{a\rightarrow\infty}
\frac{
\|(\Mu{\mathbb{Z}}f_a)^{(k)}\|_{\ell^p([2N,(a-2)N-k-1])}^p
}{
\|(\Mu{\mathbb{Z}_N}f)^{(k)}\|_{\ell^p(\mathbb{Z}_N)}^p
}
&\geq1
,&
\limsup_{a\rightarrow\infty}
\|(\Mu{\Z}f_a)^{(k)}\|_{\ell^p(\mathbb{Z}\setminus[2N,(a-2)N-k-1])}^p
&<\infty
\end{align*}
which implies
\[
\liminf_{a\to\infty}
\frac{
\|\MZu f_a^{(k)}\|_{\ell^p(\mathbb{Z})}^p
}{
a\|(\Mu{\mathbb{Z}_N}f)^{(k)}\|_{\ell^p(\mathbb{Z}_N)}^p
}
\geq1
,
\]
finishing the proof.
\end{proof}

\section{Proof of the Main Results}

Our results \zcref{thm:c0p,thm:c1p,thm:czu1p_c1p,thm:c2p,thm:ckp,thm:ckpgeq} can be divided into two groups:
In \zcref{thm:c0p,thm:c1p,thm:c2p} the main focus is on upper bounds, while in \zcref{thm:czu1p_c1p,thm:ckp,thm:ckpgeq} it is on lower bounds.
We will first prove the upper and then the lower bounds.

\subsection{Proof of \texorpdfstring{\zcref{thm:c0p}}{\ref{thm:c0p}}}\label{sec:proof:it:main theorem c0p}
This result and proof are along the same lines as \zcref{operator_norm_of_uncentered_maximal_operator_in_continuous_setting}, the corresponding result for $f\in L^p(\R)$ from \cite{GM3}.

\begin{proof}[Proof of \zcref{thm:c0p}]
For a set of intervals $\mathcal I$, we denote by $\bigcup\mathcal I$ the union of all intervals in $\mathcal I$.
For $\lambda>0$, let $\mathcal I_\lambda$ be a finite set of open intervals $I$ with $\lm{I\cap S}\geq\lambda\lm{I}$.
Following the argument in \cite[Lemma~4.4]{G}, there exists a subcollection \(\{I_1,I_2,\dotsc,I_n\}=\mathcal{I}_\lambda'\subset\mathcal{I}_\lambda\) with $\bigcup\mathcal{I}_\lambda'=\bigcup\mathcal{I}_\lambda$
and such that 
with $I_i=(a_i,b_i)$ we have
\(
b_{i-1}
<b_{i-1}
<a_{i+1}
<a_{i+2}
.
\)
Define
\begin{align*}
\mathcal{I}_\lambda^0
&:=
\{I_i\in\mathcal{I}_\lambda': i\text{ is even}\}
,&
\mathcal{I}_\lambda^1
&:=
\{I_i\in\mathcal{I}_\lambda': i\text{ is odd}\}
.
\end{align*}
Then, for each $i=0,1$ the intervals in $\mathcal{I}_\lambda^i$ are pairwise disjoint.

Since
\[
\left|\bigcup\mathcal{I}_\lambda^i\right|
=
\sum_{I\in\mathcal{I}_\lambda^i}\lm{I}
\leq
\sum_{I\in\mathcal{I}_\lambda^i}\lm{I\cap S}/\lambda
=
\lm{S\cap\left(\bigcup\mathcal{I}_\lambda^i\right)}/\lambda
\]
we have
\begin{align*}
\lm{\bigcup\mathcal{I}_\lambda}&=\lm{
\bigcup\mathcal{I}_\lambda^0}+\lm{\bigcup\mathcal{I}_\lambda^1}
-
\lm{
\left(\bigcup\mathcal{I}_\lambda^{0}\right)
\cap
\left(\bigcup\mathcal{I}_\lambda^{1}\right)
}
\\
&\leq
\lm{
S\cap\bigcup\mathcal{I}_\lambda^{0}
}
/\lambda
+
\lm{
S\cap\bigcup\mathcal{I}_\lambda^{1}
}
/\lambda
-
\lm{
S\cap
\bigcup\mathcal{I}_\lambda^{0}
\cap
\bigcup\mathcal{I}_\lambda^{1}
}
.
\end{align*}
Since $\lambda\leq1$, this last inequality is maximized if $S\subset\bigcup\mathcal{I}_\lambda^{i}$ for $i=0,1$ in which case it equals $(2/\lambda-1)\lm{S}$.

By the Lindelöf property, there is a countable set of intervals $I$ with $\lm{I\cap S}>\lambda$ whose union equals $\{\MZu\chi_S>\lambda\}$.
This countable union can, in turn, be exhausted by an increasing sequence of finite unions of intervals. In conclusion, by the previous display, and since $\|\chi_S\|_{\ell^p(\Z)}^p=\cm S$, for all $0<\lambda\leq 1$ we obtain
\begin{equation}
\label{superlevelestimate}
\frac{
\cm{\{\MZu\chi_S>\lambda\}}
}{
\|\chi_S\|_{\ell^p(\Z)}^p
}
\leq
\frac2\lambda-1
.
\end{equation}
Note that with $S_N=[0,N-1]$, for any $0<\lambda<1$ we have
\[
\{\M\chi_{S_N}>\lambda\}
=
\left[\lceil N-N/\lambda\rceil,\lfloor N/\lambda\rfloor\right]
.
\]
This means for $S=S_N$ \zcref{superlevelestimate} turns into an equality as $N\rightarrow\infty$.

In particular,
\[
\frac{
\lambda\cm{\{\MZu\chi_S>\lambda\}}
}{
\|\chi_S\|_{\ell^p(\Z)}^p
}
\leq
2-\lambda
\leq
2
,
\]
and the inequalities become equalities for $S=\{0\}$ as $\lambda\rightarrow0$.
This implies \(C_{0,p,\infty}=2^{1/p}\).

To compute $C_{0,p}$, by the layer cake formula and \zcref{superlevelestimate},
\[
\frac{
\|\MZu\chi_S\|_{\ell^p(\Z)}^p
}{
\|\chi_S\|_{\ell^p(\Z)}^p
}
=
\int_0^1
p\lambda^{p-1}
\frac{
\cm{\{\MZu\chi_S>\lambda\}}
}{
\|\chi_S\|_{\ell^p(\Z)}^p
}
\intd\lambda
\leq
\int_0^1
p\lambda^{p-1}
\left(\frac2\lambda-1\right)
\intd\lambda
=
\frac{2p}{p-1}-1
=
\frac{p+1}{p-1}
.
\]
Again, inserting $S=S_N$ as above and letting $N\rightarrow\infty$, by Fatou's lemma, the above inequality turns into an equality, giving $C_{0,p}=\left(\frac{p+1}{p-1}\right)^{1/p}$.
\end{proof}

\subsection{Proof of \texorpdfstring{\zcref{thm:c1p}}{\ref{thm:c1p}}}\label{sec:proof:it:main theorem c1p}

Recall that for $1\leq p<\infty$, with
\[\kappa_p:=\sum_{n=1}^\infty {\left(\frac{1}{n}-\frac{1}{n+1}\right)}^p\]
we have to prove
\[
\sum_{n\in\Z}
|{(\MZu\chi_S)}'(n)|^p
\leq
\kappa_p
\sum_{n\in\Z}
|\chi_S'(n)|^p
\]
for all $S\subset\Z$.
We first prove the following local bounds:


\begin{lemma}
\label{lem:firstderivativeoninterval}
Let $S\subset\Z$ and $a<b$ with $S\cap[a,b]=\{a\}$ and assume $\MZu\chi_S$ is non-increasing on $[a,b]$.
Then
\[
\sum_{n=a}^{b-1}
|{(\MZu\chi_S)}'(n)|^p
\leq
\kappa_p
=
\kappa_p
|\chi_S'(a)|^p
.
\]
\end{lemma}

\begin{proof}
Define two finite sequences ${\{s_n\}}_{n=1}^{N_s},{\{t_n\}}_{n=1}^{N_t}$ by
\begin{align*}
s_n
&=
\begin{cases}
\frac{1}{n}-\frac{1}{n+1}&\text{if }\frac{1}{n+1}>\MZu\chi_S(b)\\
\frac{1}{n}-\MZu\chi_S(b)&\text{if }\frac{1}{n+1}\leq \MZu\chi_S(b)<\frac{1}{n}\\
0&\text{if }n\leq b-a\text{ and }\frac{1}{n}\leq \MZu\chi_S(b)
,
\end{cases}
\\
t_n
&=
\MZu\chi_S(a+n-1)-\MZu\chi_S(a+n)
\end{align*}
if $a+n\leq b$. Then, we have for any $k$ with $k<b-a$, we have
\[\sum_{n=1}^ks_n=1-\frac{1}{k+1}\geq1-\MZu\chi_S(a+k)=\sum_{n=1}^kt_n\]
and 
\[\sum_ns_n=1-\MZu\chi_S(b)=\sum_nt_n.\]
Therefore, by \zcref{Karamata} (Karamata's inequality), we can conclude
\[
\sum_{n=a}^{b-1}{|{(\MZu\chi_S)}'(n)|}^p\leq\kappa_p
.
\]
\end{proof}
By symmetry, we also have the following corollary:
\begin{corollary}\label{cor:firstderivativeoninterval}
Let $S\subset\Z$ and $a<b$ with $S\cap[a,b]=\{b\}$ and assume $\MZu\chi_S$ is non-decreasing on $[a,b]$.
Then
\[
\sum_{n=a}^{b-1}
|{(\MZu\chi_S)}'(n)|^p
\leq
\kappa_p
=
\kappa_p
|\chi_S'(b-1)|^p
.
\]
\end{corollary}

\begin{proof}[Proof of \zcref{thm:c1p}]
The case $p=\infty$ is a consequence of \zcref{flat_bound} and ${\|\chi_S'\|}_{\ell^\infty(\Z)}=1$. It remains to consider $1\leq p<\infty$.
By \zcref{lem:c1pboundedsuffices}, it suffices to consider nonempty bounded $S$. Thus, there exists $m\geq1$ and $a_1\leq b_1<a_2\leq b_2<\ldots<a_m\leq b_m$ such that
\[S=[a_1,b_1]\cup\dotsb\cup [a_m,b_m].\] 
That means we may decompose $\mathbb{Z}$ into intervals $I$ of four different types: $(-\infty,a_1-1]$, $[a_i,b_i-1]$, $[b_i,a_{i+1}-1]$ and $[b_m,\infty)$.
In order to prove $C_{1,p}\leq\kappa_p^{\frac1p}$, on each of these intervals $I$, we will show
\[
\sum_{n\in I}
|{(\MZu\chi_S)}'(n)|^p
\leq
\kappa_p
\sum_{n\in I}
|\chi_S'(n)|^p
,
\]
using \zcref{lem:firstderivativeoninterval,cor:firstderivativeoninterval}.

\paragraph{\underline{$I=(-\infty,a_1-1]$:}} 
By \zcref{cor:firstderivativeoninterval}, for every $N\in\N$ we have
\[
\sum_{n=-N-1}^{a_1-1}{|{(\MZu\chi_S)}'(n)|}^p
\leq
\kappa_p
{|\chi_S'(a_1-1)|}^p
\]
and we let $N\rightarrow\infty$ to finish this case.

\paragraph{\underline{$I=[b_m,\infty)$:}} 
This follows from \zcref{lem:firstderivativeoninterval} similarly to the previous case.

\paragraph{\underline{$I=[a_i,b_i-1]$:}}
On this interval we have $(\MZu\chi_S)'=0$ since $\MZu\chi_S=1$ on $[a_i,b_i]$.

\paragraph{\underline{$I=[b_i,a_{i+1}-1]$:}}
By \zcref{cor: convexity} there is a $c_i\in [b_i,a_{i+1}]$ such that $\MZu\chi_S$ is decreasing on $[b_i,c_i]$ and increasing on $[c_i,a_{i+1}]$.
Thus, by applying \zcref{lem:firstderivativeoninterval,cor:firstderivativeoninterval} to these two intervals we get
\[
\sum_{n=b_i}^{a_{i+1}-1}{|{(\MZu\chi_S)}'(n)|}^p
\leq
\kappa_p(
|\chi_S'(b_i)|^p
+
|\chi_S'(a_{i+1}-1)|^p
)
.
\]
Combining these four cases, we obtain the desired upper bound:
\[\sum_{n\in\mathbb{Z}}{|{(\MZu\chi_S)}'(n)|}^p\leq\kappa_p\sum_{n\in\mathbb{Z}}{|\chi_S'(n)|}^p.\]
Thus, $C_{1,p}\leq\kappa_p^{1/p}$.
\par
As for the lower bound of $C_{1,p}$, take $S=\{0\}$.
Then out of the above cases we only have the intervals $(-\infty,a_1-1]$ and $[b_m,\infty)$.
Moreover, $\MZu\chi_S(n)=\frac{1}{1+|n|}$, which means all inequalities become equalities as $N\rightarrow\infty$ and hence \(C_{1,p}\geq\kappa_p^{\frac1p}\).
\end{proof}


\subsection{Proof of \texorpdfstring{\zcref{thm:c2p}}{\ref{thm:c2p}}}\label{sec:proof:it:main theorem c2p}

In this \zcref[noref,nocap]{sec:proof:it:main theorem c2p} we prove the upper bound $C_{2,p}\leq\frac{1}{2}$.
The lower bound follows from \zcref{thm:ckp}, which is proven in \zcref{sec:proof:it:main theorem ckp}.

The proof relies on the following pointwise bounds on $(\MZu f)''$, \zcref{lem:m2fleqf22,lem:m2fleq12}:

\begin{lemma}
\label{lem:m2fleqf22}
Let $S\subset\Z$.
Let $n\in\Z$ such that $|\chi_S''(n)|=2$ or $\chi_S(n+1)=1$.
Then $|(\MZu\chi_S)''(n)|\leq\frac{|\chi_S''(n)|}{2}$.
\end{lemma}

\begin{proof}
The case $|\chi_S''(n)|=2$ is a consequence of \zcref{flat_bound}, so it remains to consider $\chi_S(n+1)=1$ which implies also $\MZu \chi_S(n+1)=1$.
In this case $\chi_S''(n)=\chi_S(n)+\chi_S(n+2)-2\leq0$ which means $\chi_S''(n)\in\{0,-1\}$.

If $\chi_S''(n)=0$ then $\chi_S(n)=\chi_S(n+2)=1$, and thus also $\MZu \chi_S=1$ on $\{n,n+1,n+2\}$ which means
\(
|(\MZu \chi_S)''(n)|=0=\frac{1}{2}|\chi_S''(n)|.
\)

If $\chi_S''(n)=-1$ then $\{\chi_S(n),\chi_S(n+2)\}=\{0,1\}$ which means $\max\{\MZu \chi_S(n),\MZu \chi_S(n+2)\}=1$.
Since $\chi_S(n+1)=1$ we have $\min\{\MZu \chi_S(n),\MZu \chi_S(n+2)\}\geq\frac12$ and thus
\[
|(\MZu \chi_S)''(n)|=1-\min\{\MZu \chi_S(n),\MZu \chi_S(n+2)\}\leq \frac{1}{2}=\frac{|\chi_S''(n)|}{2}.
\]
\end{proof}

\begin{lemma}
\label{lem:m2fleq12}
Let $S\subset\Z$ be bounded. Let $n\in\Z$ such that $\chi_S(n+1)=0$ and $|\chi_S''(n)|\leq 1$.
Then $|(\MZu \chi_S)''(n)|\leq\frac{1}{2}$.
\end{lemma}

\begin{remark}
\zcref[S]{lem:m2fleq12} also holds for unbounded $S$, but for our purposes it suffices to consider the bounded case.
\end{remark}

\begin{proof}[Proof of \zcref{lem:m2fleq12}]
If $(\MZu \chi_S)'(n)$ and $(\MZu \chi_S)'(n+1)$ have the same sign then by \zcref{thm: GM2} we have
\[
|(\MZu \chi_S)''(n)|
\leq
\|(\MZu \chi_S)'\|_{\ell^\infty(\Z)}
\leq
\frac{1}{2}\|\chi_S'\|_{\ell^\infty(\Z)}
\leq
\frac{1}{2}.
\]

It remains to consider the case that they have opposite sign.
Since $\chi_S(n+1)=0$, by \zcref{lem: convexity} we have $(\MZu \chi_S)''(n)\geq0$ which in this case means $\MZu \chi_S'(n+1)>0$ and $\MZu \chi_S'(n)<0$.
Then by \zcref{lem: good balls}, for $m\in\{n,n+2\}$ there is a bounded interval $I_m\ni m$ such that for $N_m=\cm{I_m\cap S}$ and $K_m=\cm{I_m}$ we have
\(
\MZu \chi_S(m)=\frac{N_m}{K_m}
.
\)
Moreover, $n\notin I_{n+2}$ and $n+2\notin I_n$, so $I_n\cap I_{n+2}=\emptyset$.
Then we can compute
\begin{align}
\nonumber
(\MZu \chi_S)''(n)
&=
\MZu \chi_S(n)-2\MZu \chi_S(n+1)+\MZu \chi_S(n+2)
\\
\nonumber
&\leq
\frac{N_n}{K_n}-2\frac{N_n+N_{n+2}}{K_n+K_{n+2}+1}+\frac{N_{n+2}}{K_{n+2}}
\\
\label{eq:mf2formula}
&=
\frac{(K_{n+2}-K_n)(N_nK_{n+2}-N_{n+2}K_n)+N_nK_{n+2}+N_{n+2}K_n}{K_nK_{n+2}(K_n+K_{n+2}+1)}
\end{align}
Moreover, $N_m\leq K_m$ and by assumption
\[
\frac{N_n}{K_n+1}
\leq
\MZu \chi_S(n+1)
<
\frac{N_{n+2}}{K_{n+2}}
,
\]
which implies
\begin{equation}\label{ineq: N2}
N_nK_{n+2}-N_{n+2}K_n\leq N_{n+2}.
\end{equation}
By symmetry it suffices to consider the case $K_n\leq K_{n+2}$.

First consider the case $K_n\geq 2$.
Then, using \zcref{ineq: N2,eq:mf2formula}
\begin{align*}
|(\MZu \chi_S)''(n)|
&\leq
\frac{(K_{n+2}-K_n)N_{n+2}+N_nK_{n+2}+N_{n+2}K_n}{K_nK_{n+2}(K_n+K_{n+2}+1)}
\\
&=
\frac{K_{n+2}(N_n+N_{n+2})}{K_nK_{n+2}(K_n+K_{n+2}+1)}
\leq
\frac{1}{K_n}
\leq
\frac{1}2
\end{align*}

Otherwise, $K_n=1$, and thus $N_n=1$ (since $1\leq N_n\leq K_n$) and $\MZu \chi_S(n)=1$. By \zcref{lem: good balls} and the assumption that $S$ is bounded this means $\chi_S(n)=1$. Additionally, since $\chi_S(n+1)=0$ and $\chi_S''(n)\leq 1$ we conclude that $\chi_S(n+2)=0$.
Thus, $\MZu \chi_S(n+2)\neq 1$ and $K_{n+2}\geq 2$.
Moreover, by \zcref{ineq: N2} we have $N_{n+2}\geq \frac{K_{n+2}}{2}$, and therefore by \zcref{eq:mf2formula} 
\begin{align*}
    |(\MZu \chi_S)''(n)|
&=\frac{(K_{n+2}-1)(K_{n+2}-N_{n+2})+K_{n+2}+N_{n+2}}{K_{n+2}(K_{n+2}+2)}\\
&=\frac{K^2_{n+2}+N_{n+2}(2-K_{n+2})}{K_{n+2}(K_{n+2}+2)}\\
&\leq \frac{K^2_{n+2}+(K_{n+2}/2)(2-K_{n+2})}{K_{n+2}(K_{n+2}+2)}
=
\frac{1}{2}.
\end{align*}
\end{proof}

Now, we can prove \zcref{thm:c2p}:

\begin{proof}[Proof of \zcref{thm:c2p}]
Since in the proof of \zcref{thm:ckp} we already proved $C_{2,p}\geq\frac{1}{2}$ it remains to show the reverse inequality.

We first consider $p=\infty$ in which case ${\|\chi_S''\|}_{\ell^\infty(\Z)}\in\{0,1,2\}$.
If ${\|\chi_S''\|}_{\ell^\infty(\Z)}=0$ then $\chi_S$ and $\MZu \chi_S$ are constant which means ${\|(\MZu \chi_S)''\|}_{\ell^\infty(\Z)}=0=\frac12{\|\chi_S''\|}_{\ell^\infty(\Z)}$.
If ${\|\chi_S''\|}_{\ell^\infty(\Z)}=1$ then by \zcref{lem:m2fleqf22,lem:m2fleq12} for any $n\in\Z$ we have \(|{(\MZu \chi_S)}''(n)|\leq\frac{1}{2}{\|\chi_S''\|}_{\ell^\infty(\Z)}\).
If ${\|\chi_S''\|}_{\ell^\infty(\Z)}=2$ then by \zcref{flat_bound} we have ${\|(\MZu \chi_S)''\|}_{\ell^\infty(\Z)}\leq1=\frac12{\|\chi_S''\|}_{\ell^\infty(\Z)}$.

It remains to consider $1\leq p<\infty$. By \zcref{lem:c1pboundedsuffices}, we may assume that $S$ is finite. In other words, 
\[S=[a_1,b_1]\sqcup [a_2,b_2]\sqcup\dotsb\sqcup [a_m,b_m]\]
with $a_1\leq b_1\leq a_2\leq b_2\leq\dotsb\leq a_m\leq b_m$ and $b_i<a_{i+1}-1$ for all $1\leq i\leq m-1$. 
In order to bound \(\|{(\MZu\chi_S)}''\|_{\ell^p(\Z)}^p\) we decompose
$\Z$ into intervals $I$ of the form $(-\infty,a_1-2]$, $[a_i-1,b_i-1]$, $[b_i,a_{i+1}-2]$ and $[b_m,\infty)$.
On each of these intervals $I$ we prove
\begin{equation}\label{eq:c2p_bounded_by_parts}
\sum_{n\in I}
|{(\MZu\chi_S)}'(n)|^p
\leq
2^{-p}
\sum_{n\in I}
|\chi_S'(n)|^p
\end{equation}
which implies $C_{2,p}\leq\frac12$.

Recall that by \zcref{lem: convexity}, ${(\MZu \chi_S)}''(n)\geq0$ for all $n\in\mathbb{Z}\setminus(S-1)$ and ${(\MZu \chi_S)}''(n)\leq0$ for all $n\in S-1$:

\paragraph{\underline{$(-\infty,a_1-2]$}}
Since $\MZu \chi_S$ is convex and bounded on $(-\infty,a_1-2]$ we have $(\MZu \chi_S)'(n)\to0$ as $n\to-\infty$, and thus by \zcref{flat_bound} we have
\[
\sum_{n=-\infty}^{a_1-2}{|{(\MZu \chi_S)}''(n)|}^p=\sum_{n=-\infty}^{a_1-2}{{(\MZu \chi_S)}''(n)}^p\leq {\left(\sum_{n=-\infty}^{a_1-2}{(\MZu \chi_S)}''(n)\right)}^p
={(\MZu \chi_S)}'(a_1-1)^p\leq 2^{-p},
\]
while
\[\sum_{n=-\infty}^{a_1-2}{|\chi_S''(n)|}^p=\chi_S''(a_1-2)=1.\]
Thus, \zcref{eq:c2p_bounded_by_parts} holds for $I=(-\infty,a_1-2]$.

\paragraph{\underline{$[a_i-1,b_i-1]$}}
If $a_i=b_i$ then
\(
\chi_S''(a_i-1)=-2
,
\)
which means \zcref{eq:c2p_bounded_by_parts} holds by \zcref{lem:m2fleqf22}.

It remains to consider $a_i<b_i$. Then, since $\MZu \chi_S=1$ on $[a_i,b_i]\supset\{a_i,a_i+1,b_i-1,b_i\}$, and since by \zcref{flat_bound} we have $\MZu \chi_S\geq\frac12$ on $\{a_i-1,b_i+1\}$, we can conclude
\[
\sum_{n=a_i-1}^{b_i-1}{|{(\MZu \chi_S)}''(n)|}^p
=
{|{(\MZu \chi_S)}''(a_i-1)|}^p+{|{(\MZu \chi_S)}''(b_i-1)|}^p
\leq
2^{-p}+2^{-p}
,
\]
while
\begin{equation*}
\sum_{n=a_i-1}^{b_i-1}{|\chi_S''(n)|}^p={|\chi_S''(a_i-1)|}^p+{|\chi_S''(b_i-1)|}^p=2.
\end{equation*}
Thus, \zcref{eq:c2p_bounded_by_parts} holds.

\paragraph{\underline{$[b_i,a_{i+1}-2]$}}
If $b_i=a_{i+1}-2$ then
\(
\chi_S''(b_i)=2
,
\)
which means \zcref{eq:c2p_bounded_by_parts} holds by \zcref{lem:m2fleqf22}.

It remains to consider $b_i<a_{i+1}-2$.
This is the main case.
For all $n\in [b_i,a_{i+1}-2]$, we have $|\chi_S''(n)|\leq1$ and $\chi_S(n+1)=0$ which by \zcref{lem:m2fleq12} implies $|{(\MZu \chi_S)}''(n)|\leq\frac{1}{2}$.
Moreover, by \zcref{lem: convexity} we have ${(\MZu \chi_S)}''\geq0$ on $[b_i,a_i-2]$ and together with \zcref{flat_bound} we can conclude
\[
\sum_{n=b_i}^{a_{i+1}-2}{|{(\MZu \chi_S)}''(n)|}^p
\leq\frac{1}{2^{p-1}}
\sum_{n=b_i}^{a_{i+1}-2}{{(\MZu \chi_S)}''(n)}
=
\frac{
{(\MZu \chi_S)}'(a_{i+1}-1)
-
{(\MZu \chi_S)}'(b_i)
}{
2^{p-1}
}
\leq\frac{1}{2^{p-1}}
,
\]
while
\[\sum_{n=b_i}^{a_{i+1}-2}{|\chi_S''(n)|}^p={|\chi_S''(b_i)|}^p+{|\chi_S''(a_{i+1}-2)|}^p=2.\]
Thus, \zcref{eq:c2p_bounded_by_parts} holds.

\paragraph{\underline{$[b_m,\infty)$}}
This case follows from the same argument as $I=(-\infty,a_1-2]$ by symmetry.

We have proven \zcref{eq:c2p_bounded_by_parts} for all intervals in our decomposition of $\Z$, which finishes the proof of $C_{2,p}\leq\frac{1}{2}$.
Together with \zcref{thm:ckp}, which is shown in \zcref{sec:proof:it:main theorem ckp}, we get the desired conclusion $C_{2,p}=\frac{1}{2}$ for all $1\leq p\leq\infty$.
\end{proof}

\subsection{Proof of \texorpdfstring{\zcref{thm:czu1p_c1p}}{\ref{thm:czu1p_c1p}}}\label{sec:proof:it: czu1p_c1p}

Recall that by \zcref{thm:c1p}, we have
\[C_{1,p}^p=\sum_{n=1}^\infty{\left(\frac{1}{n}-\frac{1}{n+1}\right)}^p.\]

The following function will serve as the example for all $1<p<\infty$:
\begin{align*}
f(n)
&:=
\max\{0,2-|n|\}
=
\begin{cases}
2&\text{if }n=0
,\\
1&\text{if }n=\pm1
,\\
0&\text{else}
\end{cases}
&
\MZu f(n)
&=
\begin{cases}
2&\text{if }n=0,\\
\frac{3}{2}&\text{if }n=\pm1,\\
\frac{4}{|n|+2}&\text{else}
\end{cases}
\end{align*}
and
\begin{align*}
{\|{(\MZu f)}'\|}_{\ell^p(\mathbb{Z})}^p
&=
\sum_{n\in\mathbb{Z}}{|\MZu f(n+1)-\MZu f(n)|}^p
=
2\left(\frac{1}{2^p}+\frac{1}{2^p}+\sum_{n=2}^\infty{\left|\frac{4}{n+3}-\frac{4}{n+2}\right|}^p\right)
\\
{\|f'\|}_{\ell^p(\mathbb{Z})}^p
&=
\sum_{n\in\mathbb{Z}}{|f(n+1)-f(n)|}^p={|1-0|}^p+{|2-1|}^p+{|1-2|}^p+{|0-1|}^p=4
.
\end{align*}
Thus
\begin{align*}
\frac{{\|{(\MZu f)}'\|}_{\ell^p(\mathbb{Z})}^p}{{\|f'\|}_{\ell^p(\mathbb{Z})}^p}-C_{1,p}^p
&=
\frac{1}{2^p}+\frac12\sum_{n=4}^\infty{\left|\frac{4}{n+1}-\frac{4}n\right|}^p
-
\sum_{n=1}^\infty{\left(\frac{1}{n}-\frac{1}{n+1}\right)}^p
\\
&=
\frac12\sum_{n=4}^\infty{\left|\frac{4}{n+1}-\frac{4}n\right|}^p
-
\sum_{n=2}^\infty{\left(\frac{1}{n}-\frac{1}{n+1}\right)}^p
.
\end{align*}
Hence, it suffices to prove
\begin{equation}\label{eq:comparison_of_series}
\sum_{n=4}^\infty {\left(\frac{4}{n}-\frac{4}{n+1}\right)}^p>2\sum_{n=2}^\infty {\left(\frac{1}{n}-\frac{1}{n+1}\right)}^p
\end{equation}
in order to obtain the desired conclusion, that for all $1<p<\infty$, we have
\(\CZu_{1,p}>C_{1,p}.\)
For $n\in\mathbb{N}$ define
\begin{align*}
s_n
&=
\frac{4}{n+3}-\frac{4}{n+4}
&
t_n
&=
\frac{1}{\lceil n/2\rceil+1}-\frac{1}{\lceil n/2\rceil+2}
.
\end{align*}
Then, \zcref{eq:comparison_of_series} is equivalent to 
\begin{equation}\label{eq:x_n^p>y_n^p}
\sum_{n=1}^\infty s_n^p>\sum_{n=1}^\infty t_n^p.
\end{equation}
To prove this, we need to use \zcref{Karamata}. Once the sequences ${\{s_i\}}_{i=1}^n$ and ${\{t_i\}}_{i=1}^n$ satisfy the desired conditions for all $n\in\mathbb{N}$, then we will obtain \zcref{eq:comparison_of_series} and thus complete the proof. 
\par
We check the conditions one-by-one:
Clearly, we have $s_n\geq s_{n+1}$ and $t_n\geq t_{n+1}$ for all $n$.
Moreover, notice that 
\[\sum_{i=1}^ns_i=1-\frac{4}{n+4}.\]
As for the right hand side, if $n=2m$ is even, then 
\[\sum_{i=1}^nt_i=2\sum_{i=1}^m\frac{1}{i+1}-\frac{1}{i+2}=1-\frac{2}{m+2}=1-\frac{4}{n+4};\]
On the other hand, if $n=2m+1$ is odd, then
\begin{align*}
\sum_{i=1}^nt_i
&=
\left(2\sum_{i=1}^m\frac{1}{i+1}-\frac{1}{i+2}\right)+\left(\frac{1}{m+2}-\frac{1}{m+3}\right)
=
1-\frac{1}{m+2}-\frac{1}{m+3}
=
1-\frac{2}{n+3}-\frac{2}{n+5}
\\
&=
1-\frac{4(n+4)}{(n+3)(n+5)}<1-\frac4{n+4}
.
\end{align*}
Since the last inequality is strict and the map $t\mapsto t^p$ on $[0,\infty)$ is strictly convex for $1<p<\infty$ we can conclude \zcref{eq:x_n^p>y_n^p} from \zcref{Karamata,infinite_Karamata} and finish the proof.
Letting $n\to\infty$, we obtain
\[\sum_{i=1}^\infty s_i=\sum_{i=1}^\infty t_i=1.\]
To sum up, we finish the proof of \zcref{eq:comparison_of_series} and thus the proof of \zcref{thm:czu1p_c1p}.

\subsection{Proof of \texorpdfstring{\zcref{thm:ckp}}{\ref{thm:ckp}}}\label{sec:proof:it:main theorem ckp}

\begin{proof}[Proof of \zcref{thm:ckp}]
For $\chi_{\{0\}}:\mathbb{Z}_3\rightarrow\{0,1\}$ we have \(\Mu{\mathbb{Z}_3}\chi_{\{0\}}=\frac12\chi_{\{0\}}+\frac12\), which means that for all $k\geq1$ we have
\(
{(\Mu{\mathbb{Z}_3}\chi_{\{0\}})}^{(k)}
=
\frac12
\chi_{\{0\}}^{(k)}
.
\)
Let $S_a=\{0,3,6,\dotsc,3a\}$. Therefore, by \zcref{lem:discretevsperiodic}, we obtain that 
\[C_{k,p}\geq\lim_{a\to\infty}\frac{{\left\|{(\MZu \chi_{S_a})}^{(k)}\right\|}_{\ell^p(\mathbb{Z})}}{{\left\|\chi_{S_a}^{(k)}\right\|}_{\ell^p(\mathbb{Z})}}\geq\frac{{\left\|{(\Mu{\mathbb{Z}_3}\chi_{\{0\}})}^{(k)}\right\|}_{\ell^p(\mathbb{Z}_3)}}{{\left\|\chi_{\{0\}}^{(k)}\right\|}_{\ell^p(\mathbb{Z}_3)}}=\frac{1}{2}.\]
This finishes the proof of \zcref{thm:ckp} and in particular $C_{2,p}\geq\frac{1}{2}$, the remaining inequality in \zcref{thm:c2p}.
\end{proof}

\subsection{Proof of \texorpdfstring{\zcref{thm:ckpgeq}}{\ref{thm:ckpgeq}}}\label{sec:proof:it:main theorem ckpgeq}

On the circle $\mathbb{Z}_N$, for a function $f:\mathbb{Z}_N\to\mathbb{R}$, we define its discrete Fourier transform as follows:
\[\ft{f}(\xi):=\sum_{n\in\mathbb{Z}_N}f(n)e^{-2\pi i\frac{n}{N}\xi}.\]
Then,
\begin{align*}
\|f\|_{\ell^2(\Z_N)}^2
&=
\frac1N
\|\ft f\|_{\ell^2(\Z_N)}^2
,&
\ft{f^{(k)}}(\xi)
&=
(
e^{2\pi i\frac{\xi}{N}}
-1
)^k
\ft f(\xi)
.
\end{align*}
\begin{proof}[Proof of \zcref{thm:ckpgeq}]
Now, let $N=8$.
Then
\(
|
e^{2\pi i\frac \xi 8}
-1
|
\)
is maximal for $\xi =4$ where it assumes the value $2$,
while for $\xi =0,\ldots,7$ with $\xi \neq 4$ we have
\[
|
e^{2\pi i\frac \xi 8}
-1
|
\leq
|
e^{2\pi i\frac 38}
-1
|
=
|
1
+
e^{\frac{\pi i}4}
|
<
2
.
\]
Set $f=\chi_{\{0,1,2,5\}}$.
Interpreting vectors as maps $\Z_8\to\R$ so that $f=(1,1,1,0,0,1,0,0)$ we have $\Mu{\mathbb{Z}_8} f=(1,1,1,\frac34,\frac23,1,\frac23,\frac34)$.
This means
\(
\ft f(4)=0
\)
while
\[\ft{\Mu{\mathbb{Z}_8} f}(4)=\sum_{n=0}^7{(-1)}^{n+1}\Mu{\mathbb{Z}_8}f(n)=-\frac{1}{6}\neq0,\]
and thus
\begin{align*}
8\|f^{(k)}\|_{\ell^2(\Z_8)}^2
&=
\|\ft{f^{(k)}}\|_{\ell^2(\Z_8)}^2
\leq
{\|\ft{f}\|}_{\ell^2(\Z_8)}^2|1+e^{\frac{\pi i}4}|^{2k}
=
4
|1+e^{\frac{\pi i}4}|^{2k},
\\
8
\|(\Mu{\mathbb{Z}_8} f)^{(k)}\|_{\ell^2(\Z_8)}^2
&=
\|\ft{(\Mu{\mathbb{Z}_8} f)^{(k)}}\|_{\ell^2(\Z_8)}^2
\geq
|\ft{\Mu{\mathbb{Z}_8} f}(4)|^2
2^{2k}
=
\frac{2^{2k}}{6^2}
,
\end{align*}
which means
\begin{equation}
\label{eq:mfkfklowerbound}
\frac{\|(\Mu{\mathbb{Z}_8} f)^{(k)}\|_{\ell^2(\Z_8)}}{\|f^{(k)}\|_{\ell^2(\Z_8)}}\geq\frac{1}{12}
\left(\frac{2}{|1+e^{\frac{\pi i}4}|}\right)^k.
\end{equation}
Now, for all $1\leq p\leq q\leq\infty$ we have \(\|f\|_{\ell^q(\Z_8)}\leq\|f\|_{\ell^p(\Z_8)}\leq8^{\frac1p-\frac1q}\|f\|_{\ell^q(\Z_8)}\).
This finishes the proof by an application of \zcref{lem:discretevsperiodic}.

Note that $\ft f(3)\neq0$, and thus the lower bound \zcref{eq:mfkfklowerbound} on the exponential growth is optimal for this particular $f$ in the sense that it also holds as an upper bound albeit with a different factor than $\frac1{12}$.
\end{proof}

\section{Acknowledgments.}
We would like to thank Emanuel Carneiro for discussions relating continuous and discrete higher derivatives.

J.M.\ was partially supported by the AMS Stefan Bergman Fellowship and the Simons Foundation Grant $\# 453576$.
J.W.\ was funded by the European Union's Horizon research and innovation programme under the Marie Skłodowska-Curie Action 2023 No 101151034 (SRMF).
\printbibliography

@Article{AP,
	Author = {Aldaz, J. M. and P{\'e}rez L{\'a}zaro, J.},
	Title = {Functions of bounded variation, the derivative of the one dimensional maximal function, and applications to inequalities},
	FJournal = {Transactions of the American Mathematical Society},
	Journal = {Trans. Am. Math. Soc.},
	ISSN = {0002-9947},
	Volume = {359},
	Number = {5},
	Pages = {2443--2461},
	Year = {2007},
	Language = {English},
	DOI = {10.1090/S0002-9947-06-04347-9},
	zbMATH = {5120644},
	Zbl = {1143.42021}
}

@article{BCHP,
 author = {Bober, Jonathan and Carneiro, Emanuel and Hughes, Kevin and Pierce, Lillian B.},
 title = {On a discrete version of {Tanaka}'s theorem for maximal functions},
 fjournal = {Proceedings of the American Mathematical Society},
 journal = {Proc. Am. Math. Soc.},
 issn = {0002-9939},
 volume = {140},
 number = {5},
 pages = {1669--1680},
 year = {2012},
 language = {English},
 doi = {10.1090/S0002-9939-2011-11008-6},
 zbMATH = {6028813},
 Zbl = {1245.42017}
}

@article{BW,
 author = {Bilz, Constantin and Weigt, Julian},
 title = {The one-dimensional centred maximal function diminishes the variation of indicator functions},
 fjournal = {The Journal of Geometric Analysis},
 journal = {J. Geom. Anal.},
 issn = {1050-6926},
 volume = {35},
 number = {9},
 pages = {25},
 note = {Id/No 256},
 year = {2025},
 language = {English},
 doi = {10.1007/s12220-025-01965-x},
 zbMATH = {8082734}
}

@article{CS,
 author = {Carneiro, Emanuel and Svaiter, Benar F.},
 title = {On the variation of maximal operators of convolution type},
 fjournal = {Journal of Functional Analysis},
 journal = {J. Funct. Anal.},
 issn = {0022-1236},
 volume = {265},
 number = {5},
 pages = {837--865},
 year = {2013},
 language = {English},
 doi = {10.1016/j.jfa.2013.05.012},
 zbMATH = {6261787},
 Zbl = {1317.42017}
}

@book{D,
 author = {Duoandikoetxea, Javier},
 title = {Fourier analysis. {Transl}. from the {Spanish} and revised by {David} {Cruz}-{Uribe}},
 fseries = {Graduate Studies in Mathematics},
 series = {Grad. Stud. Math.},
 issn = {1065-7339},
 volume = {29},
 isbn = {0-8218-2172-5},
 year = {2001},
 publisher = {Providence, RI: American Mathematical Society (AMS)},
 language = {English},
 zbMATH = {1557107},
 Zbl = {0969.42001}
}

@article {Ku,
    AUTHOR = {Kurka, Ond\v{r}ej},
     TITLE = {On the variation of the {H}ardy-{L}ittlewood maximal function},
   JOURNAL = {Ann. Acad. Sci. Fenn. Math.},
  FJOURNAL = {Annales Academi\ae  Scientiarum Fennic\ae . Mathematica},
    VOLUME = {40},
      YEAR = {2015},
    NUMBER = {1},
     PAGES = {109--133},
      ISSN = {1239-629X},
   MRCLASS = {42B25 (42B30 46E35)},
  MRNUMBER = {3310075},
MRREVIEWER = {Vladimir D. Stepanov},
       DOI = {10.5186/aasfm.2015.4003},
       URL = {https://doi.org/10.5186/aasfm.2015.4003},
}

@book{G,
 author = {Garnett, John B.},
 title = {Bounded analytic functions},
 edition = {Revised 1st ed.},
 fseries = {Graduate Texts in Mathematics},
 series = {Grad. Texts Math.},
 issn = {0072-5285},
 volume = {236},
 isbn = {0-387-33621-4},
 year = {2006},
 publisher = {New York, NY: Springer},
 language = {English},
 doi = {10.1007/0-387-49763-3},
 zbMATH = {5062981},
 Zbl = {1106.30001}
}

@article{GM2,
 author = {Gonz{\'a}lez-Riquelme, Cristian and Madrid, Jos{\'e}},
 title = {Sharp inequalities for maximal operators on finite graphs. {II}},
 fjournal = {Journal of Mathematical Analysis and Applications},
 journal = {J. Math. Anal. Appl.},
 issn = {0022-247X},
 volume = {506},
 number = {2},
 pages = {27},
 note = {Id/No 125647},
 year = {2022},
 language = {English},
 doi = {10.1016/j.jmaa.2021.125647},
 zbMATH = {7412837},
 Zbl = {1475.42033}
}

@article{GM3,
 author = {Grafakos, Loukas and Montgomery-Smith, Stephen},
 title = {Best constants for uncentered maximal functions},
 fjournal = {Bulletin of the London Mathematical Society},
 journal = {Bull. Lond. Math. Soc.},
 issn = {0024-6093},
 volume = {29},
 number = {1},
 year = {1997},
 language = {English},
 doi = {10.1112/S0024609396002081},
 zbMATH = {936588},
 Zbl = {0865.42020}
}

@article{Te2,
 author = {Temur, F.},
 title = {The second derivative of the discrete {Hardy}-{Littlewood} maximal function},
 fjournal = {Istanbul Journal of Mathematics},
 journal = {Istanb. J. Math.},
 issn = {2980-3020},
 volume = {3},
 number = {2},
 pages = {45--47},
 year = {2025},
 language = {English},
 doi = {10.26650/ijmath.2025.00026},
 zbMATH = {8153173}
}

@misc{TO,
 author = {Temur, Faruk and {\"O}zcan, Hikmet Burak},
 title = {The higher regularity of the discrete Hardy-Littlewood maximal function},
 year = {2025},
 howpublished = {Preprint, {arXiv}:2504.13019 [math.{CA}] (2025)},
 url = {https://arxiv.org/abs/2504.13019},
 arXiv = {arXiv:2504.13019}
}

@misc{W,
 author = {Weigt, Julian},
 title = {Sobolev bounds and counterexamples for the second derivative of the maximal function in one dimension},
 year = {2024},
 howpublished = {Preprint, {arXiv}:2409.12631 [math.{CA}] (2024)},
 url = {https://arxiv.org/abs/2409.12631},
 arXiv = {arXiv:2409.12631}
}

@article{hardy1930maximal,
  title={A maximal theorem with function-theoretic applications},
  author={Hardy, Godfrey Harold and Littlewood, John Edensor},
  journal={Acta Mathematica},
  volume={54},
  number={1},
  pages={81--116},
  year={1930},
  publisher={Springer}
}

@article{kinnunen1997hardy,
  title={The Hardy-Littlewood maximal function of a Sobolev function},
  author={Kinnunen, Juha},
  journal={Israel Journal of Mathematics},
  volume={100},
  number={1},
  pages={117--124},
  year={1997},
  publisher={Springer}
}

@software{discretemf,
  author = {Weigt, Julian},
  title = {discrete-mf},
  url = {https://codeberg.org/julian-weigt/discrete-mf},
  date = {2026-07-11}
}

@book{GrahamKnuthPatashnik1994ConcreteMathematics2e,
  author    = {Graham, Ronald L. and Knuth, Donald E. and Patashnik, Oren},
  title     = {Concrete Mathematics: A Foundation for Computer Science},
  edition   = {2},
  publisher = {Addison-Wesley},
  address   = {Reading, Massachusetts},
  year      = {1994},
  isbn      = {978-0201558029},
  note      = {See Section~2.6 (Finite and Infinite Calculus).}
}

@article{BRS,
 author = {Beltran, David and Ramos, Jo{\~a}o Pedro and Saari, Olli},
 title = {Regularity of fractional maximal functions through {Fourier} multipliers},
 fjournal = {Journal of Functional Analysis},
 journal = {J. Funct. Anal.},
 issn = {0022-1236},
 volume = {276},
 number = {6},
 pages = {1875--1892},
 year = {2019},
 language = {English},
 doi = {10.1016/j.jfa.2018.11.004},
 zbMATH = {7011439},
 Zbl = {1422.42012}
}

@article{aldaz2012optimal,
  title={Optimal bounds on the modulus of continuity of the uncentered Hardy--Littlewood maximal function},
  author={Aldaz, JM and Colzani, L and P{\'e}rez L{\'a}zaro, J},
  journal={Journal of geometric analysis},
  volume={22},
  number={1},
  pages={132--167},
  year={2012},
  publisher={Springer}
}

@article{melas2003best,
  title={The best constant for the centered Hardy-Littlewood maximal inequality},
  author={Melas, Antonios D},
  journal={Annals of mathematics},
  pages={647--688},
  year={2003},
  publisher={JSTOR}
}

@book{stein1993harmonic,
  title={Harmonic analysis: real-variable methods, orthogonality, and oscillatory integrals},
  author={Stein, Elias M and Murphy, Timothy S},
  number={43},
  year={1993},
  publisher={Princeton University Press}
}

@article {Riesz28,
    AUTHOR = {Riesz, Marcel},
     TITLE = {Sur les fonctions conjugu\'ees},
   JOURNAL = {Math. Z.},
  FJOURNAL = {Mathematische Zeitschrift},
    VOLUME = {27},
      YEAR = {1928},
    NUMBER = {1},
     PAGES = {218--244},
      ISSN = {0025-5874,1432-1823},
   MRCLASS = {99-04},
  MRNUMBER = {1544909},
       DOI = {10.1007/BF01171098},
       URL = {https://doi.org/10.1007/BF01171098},
}

@article {SW99,
    AUTHOR = {Stein, Elias M. and Wainger, Stephen},
     TITLE = {Discrete analogues in harmonic analysis. {I}. {$l^2$}
              estimates for singular {R}adon transforms},
   JOURNAL = {Amer. J. Math.},
  FJOURNAL = {American Journal of Mathematics},
    VOLUME = {121},
      YEAR = {1999},
    NUMBER = {6},
     PAGES = {1291--1336},
      ISSN = {0002-9327,1080-6377},
   MRCLASS = {42B20 (11L07 11P55 39A12 42A50 44A12)},
  MRNUMBER = {1719802},
MRREVIEWER = {Loukas\ Grafakos},
       URL =
              {http://muse.jhu.edu/journals/american_journal_of_mathematics/v121/121.6stein.pdf},
}

@misc{paper2,
author = {Liao, Sung-Yi and Madrid, Jos\'e and Palsson, Eyvindur and Weigt, Julian},
title = {Connecting regularity of discrete and continuous maximal operators},
howpublished = {In preparation},
}

\end{document}